\documentclass[11pt]{amsart}
\usepackage{amsmath,amssymb,amsthm}
\usepackage{graphicx}
\usepackage{cite,mathrsfs}
 \let\mathscr\mathcal
 \usepackage[T1]{fontenc}
 \usepackage{textcomp}
 \usepackage{lmodern}

\newtheorem{theorem}{Theorem}[section]
\newtheorem{lemma}[theorem]{Lemma}
\newtheorem{prop}[theorem]{Proposition}
\newtheorem{cor}[theorem]{Corollary}

\theoremstyle{definition}

\theoremstyle{remark}
\newtheorem{remark}[theorem]{Remark}

\numberwithin{equation}{section}

\begin{document}
%\hfill\texttt{\jobname.tex}\qquad\today

\title[Automorphic $L$-functions]{On the density function
for the value-distribution of automorphic $L$-functions}

\author{Kohji Matsumoto}
\address{K. Matsumoto: Graduate School of Mathematics, Nagoya University, Chikusa-ku, Nagoya 464-8602, Japan}
\email{kohjimat@math.nagoya-u.ac.jp}

\author{Yumiko Umegaki}
\address{{Y.\,Umegaki:} Department of Mathematical and Physical Sciences, Nara \
Women's University, Kitauoya Nishimachi, Nara 630-8506, Japan}
\email{ichihara@cc.nara-wu.ac.jp}

\thanks{
Research of the first author is
supported by Grants-in-Aid for Science Research (B) 25287002, and that of the
second author is by Grant-in-Aid for Young Scientists (B) 23740020, JSPS}

\begin{abstract}
The Bohr-Jessen limit theorem is a probabilistic limit theorem on the
value-distribution of the Riemann zeta-function in the critical strip.
Moreover their limit measure can be written as an integral involving a certain
density function.    The existence of the limit measure is now known for a quite
general class of zeta-functions, but the integral expression has been proved only for some special cases (such as Dedekind zeta-functions).
In this paper we give an alternative proof of the existence of the limit measure for
a general setting, and then prove the integral expression, with an explicitly
constructed density function, for the case of automorphic $L$-functions attached
to primitive forms with respect to congruence subgroups $\Gamma_0(N)$.
\end{abstract}

\maketitle

%\baselineskip 16pt

%%%%%%%%%%%%%%%%%%%%%%%%%%%%%%%%%%%%%%%%%%%%%%%%%%%%%%%%%%%%%%%%%%%%%%%%%%%%%%%%%%%%%
\section{Introduction}\label{sec1}
%%%%%%%%%%%%%%%%%%%%%%%%%%%%%%%%%%%%%%%%%%%%%%%%%%%%%%%%%%%%%%%%%%%%%%%%%%%%%%%%%%%%%

Let $s=\sigma+it$ be a complex variable, $\zeta(s)$ the Riemann zeta-function.
Let $R$ be a fixed rectangle in the complex plane $\mathbb{C}$, with the edges parallel to
the axes.   By $\mu_k$ we mean the $k$-dimensional usual Lebesgue measure.
For $\sigma>1/2$ and $T>0$, define
\begin{align}\label{1-1}
V_{\sigma}(T,R;\zeta)=\mu_1\{t\in[-T,T]\;|\;\log\zeta(\sigma+it)\in R\}.
\end{align}
(The rigorous definition of $\log\zeta(\sigma+it)$ will be given later, in Section \ref{sec3}.)
In their classical paper \cite{BJ3032}, Bohr and Jessen proved the existence of the limit
\begin{align}\label{1-2}
W_{\sigma}(R;\zeta)=\lim_{T\to\infty}\frac{1}{2T}V_{\sigma}(T,R;\zeta).
\end{align}
This is now called the Bohr-Jessen limit theorem.
Moreover they proved that this limit value can be written as
\begin{align}\label{1-3}
W_{\sigma}(R;\zeta)=\int_R \mathcal{M}_{\sigma}(z,\zeta)|dz|,
\end{align}
where $z=x+iy\in\mathbb{C}$, $|dz|=dxdy/2\pi$, and $\mathcal{M}_{\sigma}(z,\zeta)$ is a
continuous non-negative, explicitly constructed function defined on $\mathbb{C}$, which we may call
the density function for the value-distribution of $\zeta(s)$.

This work is a milestone in the value-distribution theory of $\zeta(s)$, and various
alternative proofs and related results have been pubished; for example, 
Jessen and Wintner \cite{JW35}, Borchsenius and Jessen \cite{BJ48}, Guo \cite{Guo96}, and
Ihara and the first author \cite{IM11}.

An important problem is to consider the generalization of the Bohr-Jessen theorem.
The first author \cite{Mat90} proved that the formula \eqref{1-2} can be generalized to a fairly
general class of zeta-functions with Euler products.
However, \eqref{1-3} has not yet been generalized to such a general class.
The reason is as follows.

The original proof of \eqref{1-2} and \eqref{1-3} by Bohr and Jessen depends on a geometric
theory of certain "infinite sums" of convex curves, developed by themselves \cite{BJ29}.
In later articles \cite{JW35} and \cite{BJ48}, the effect of the convexity of curves was
embodied in a certain inequality due to Jessen and Wintner \cite[Theorem 13]{JW35}.
Using this method, the Bohr-Jessen theory was generalized to Dirichlet
$L$-functions (Joyner \cite{Joy86}) and Dedekind zeta-functions of Galois number fields
(the first author \cite{Mat92}).    These generalizations are possible because these
zeta-functions have "convex" Euler products in the sense of \cite[Section 5]{Mat90}.
But this convexity cannot be expected for more general zeta-functions.

In \cite{Mat90}, the first author developed a method of proving \eqref{1-2} without using any
convexity, so succeeded in generalizing the theory.    However, the method in \cite{Mat90} 
cannot give a generalization of \eqref{1-3}.

So far, there is no proof of \eqref{1-3} or its analogues without using the convexity, or
the Jessen-Wintner type of inequalities.    For example, \cite{IM11} gives a different argument of
constructing the density functions for Dirichlet $L$-functions, but the argument in
\cite{IM11} also depends on the Jessen-Wintner inequality.

In \cite{Mat06} \cite{Mat07}, the first author obtained certain quantitative results on the
value-distribution of Dedekind zeta-functions of non-Galois fields and Hecke $L$-functions of
ideal class characters, whose Euler products are not convex.    But in these cases, they are
"not so far" from the case of Dedekind zeta-functions of Galois fields.    In fact, a simple generalization of the
Jessen-Wintner inequality is proved (\cite[Lemma 2]{Mat07}) and is essentially used in
the proof. 

Actually, analyzing the proof of \cite[Theorems 12, 13]{JW35} carefully, we can see that
the convexity of curves is not essential.    The indispensable tool is the inequality of
the Jessen-Wintner type.    (However the convex property is probably of independent
interest; see Section \ref{sec8}.)

It is the purpose of the present paper to obtain an analogue of \eqref{1-3} in the case of
automorphic $L$-functions.    The main result (Theorem \ref{main-1}) will be stated in
the next section.   
The key is Proposition \ref{prop-curve}, which is an
analogue of the Jessen-Wintner inequality for the automorphic case.   The novelty of this
proposition will be discussed in Section \ref{sec6}.

Except for the proof of this inequality, the argument can be carried out in more
general situation.    In Section \ref{sec3} we will introduce a general class of
zeta-functions, and in Sections \ref{sec4} to \ref{sec6} we will 
generalize the method in \cite{Mat92} to that general class.
Then in Section \ref{sec7} we will prove the Jessen-Wintner inequality for the 
automorphic case to complete the proof of the main theorem.

%%%%%%%%%%%%%%%%%%%%%%%%%%%%%%%%%%%%%%%%%%%%%%%%%%%%%%%%%%%%%%%%%%%%%%%%%%%%%%%%%%%%%%%%%%
\section{Statement of the main result}\label{sec2}
%%%%%%%%%%%%%%%%%%%%%%%%%%%%%%%%%%%%%%%%%%%%%%%%%%%%%%%%%%%%%%%%%%%%%%%%%%%%%%%%%%%%%%%%%%

Let $f$ be a primitive form of weight $\kappa$ and level $N$, 
that is a normalized Hecke-eigen new form of weight $\kappa$ with respect to the congruence
subgroup $\Gamma_0(N)$, and write its Fourier
expansion as
$$
f(z)=\sum_{n=1}^{\infty}\lambda_f(n)n^{(\kappa-1)/2}e^{2\pi inz},
%\qquad \lambda_f(1)=1.
$$
where the coefficients $\lambda_f(n)$ are real numbers with $\lambda_f(1)=1$.
Denote the associated $L$-function by
$$
L_f(s)=\sum_{n=1}^{\infty}\lambda_f(n)n^{-s}.
$$
This is absolutely convergent when $\sigma>1$, and can be continued to the whole plane
$\mathbb{C}$ as an entire function.   We understand the rigorous meaning of $\log L_f(s)$
and of
$$
V_{\sigma}(T,R;L_f)=\mu_1\{t\in[-T,T]\;|\;\log L_f(\sigma+it)\in R\}
$$
in the sense explained in Section \ref{sec3}.
The following is the main theorem of the present paper.

\begin{theorem}\label{main-1}
For any $\sigma>1/2$, the limit
\begin{align}\label{2-1}
W_{\sigma}(R;L_f)=\lim_{T\to\infty}\frac{1}{2T}V_{\sigma}(T,R;L_f)
\end{align}
exists, and can be written as
\begin{align}\label{2-2}
W_{\sigma}(R;L_f)=\int_R \mathcal{M}_{\sigma}(z,L_f)|dz|,
\end{align}
where $\mathcal{M}_{\sigma}(z,L_f)$ is a continuous non-negative function 
(explicitly given by \eqref{6-4} below) defined on $\mathbb{C}$.
\end{theorem}

The above function $\mathcal{M}_{\sigma}(w,L_f)$ can be called the density function for
the value-distribution of $L_f(s)$.    The integral expression involving the density
function is useful for quantitative studies; for example, 
in \cite{Mat92} \cite{Mat06} \cite{Mat07}
we used such expressions to evaluate the speed of convergence of \eqref{3-3} below
in the case of Dedekind zeta-functions and Hecke $L$-functions.    Therefore we may
expect that \eqref{2-2} can be used for quantitative investigation on the
value-distribution of $L_f(s)$ (see also Remark \ref{rem-quantitative}).

Let $\mathbb{P}$ be the set of all prime numbers.    Since $f$ is a common Hecke eigen 
form, $L_f(s)$ has the Euler product
\begin{align}\label{2-3}
L_f(s)&=\prod_{\substack{p\in\mathbb{P} \\ p|N}}(1-\lambda_f(p)p^{-s})^{-1}
\prod_{\substack{p\in\mathbb{P} \\ p\nmid N}}(1-\lambda_f(p)p^{-s}+p^{-2s})^{-1}\\
&=\prod_{\substack{p\in\mathbb{P} \\ p|N}}(1-\lambda_f(p)p^{-s})^{-1}
\prod_{\substack{p\in\mathbb{P} \\ p\nmid N}}(1-\alpha_f(p)p^{-s})^{-1}
(1-\beta_f(p)p^{-s})^{-1},\notag
\end{align}
where $\alpha_f(p)+\beta_f(p)=\lambda_f(p)$, $\beta_f(p)=\overline{\alpha_f(p)}$, and
\begin{align}\label{2-4}
|\alpha_f(p)|=|\beta_f(p)|=1.
\end{align}
Also we know
\begin{align}\label{2-5}
|\lambda_f(p)|\leq 1 \qquad (\mbox{if}\;\; p|N)
\end{align}
(see \cite[Theorem 4.6.17]{Miy89}).

It is known that, for any $\varepsilon>0$, there exists a set of primes 
$\mathbb{P}_f(\varepsilon)$ of positive density in $\mathbb{P}$, such that the inequality
\begin{align}\label{2-6}
|\lambda_f(p)|>\sqrt{2}-\varepsilon
\end{align}
holds for any $p\in\mathbb{P}_f(\varepsilon)$ (M. R. Murty 
\cite[Corollary 2 of Theorem 4]{Mur83} in the full
modular case, and M. R. Murty and V. K. Murty \cite[Chapter 4, Theorem 8.6]{Mur97}
for general $\Gamma_0(N)$ case).
This fact is used essentially in the course of the proof.

%%%%%%%%%%%%%%%%%%%%%%%%%%%%%%%%%%%%%%%%%%%%%%%%%%%%%%%%%%%%%%%%%%%%%%%%%%%%%%%%%%%%%%%%%%%
\section{The general formulation}\label{sec3}
%%%%%%%%%%%%%%%%%%%%%%%%%%%%%%%%%%%%%%%%%%%%%%%%%%%%%%%%%%%%%%%%%%%%%%%%%%%%%%%%%%%%%%%%%%%

A large part of the proof of our Theorem \ref{main-1} can be carried out under a more
general framework, that is, for general Euler products introduced in \cite{Mat90}.
We begin with recalling the definition of those Euler products.

Let $\mathbb{N}$ be the set of all positive integers, and $g(n)\in\mathbb{N}$, 
$f(j,n)\in\mathbb{N}$ ($1\leq j\leq g(n)$) and $a_n^{(j)}\in\mathbb{C}$.   Denote by $p_n$
the $n$-th prime number.    We assume
\begin{align}\label{3-1}
g(n)\leq C_1 p_n^{\alpha},\quad |a_n^{(j)}|\leq p_n^{\beta}
\end{align}
with constants $C_1>0$ and $\alpha,\beta\geq 0$.   Define
\begin{align}\label{3-0}
\varphi(s)=\prod_{n=1}^{\infty}A_n(p_n^{-s})^{-1},
\end{align}
where $A_n(X)$ are polynomials in $X$ given by
$$
A_n(X)=\prod_{j=1}^{g(n)}(1-a_n^{(j)}X^{f(j,n)}).
$$
Then $\varphi(s)$ is convergent absolutely in the half-plane $\sigma>\alpha+\beta+1$ by
\eqref{3-1}.    Suppose

(i) $\varphi(s)$ can be continued meromorphically to $\sigma\geq \sigma_0$, where
$\alpha+\beta+1/2\leq\sigma_0<\alpha+\beta+1$,
and all poles in this region are included in a compact subset of
$\{s\;|\;\sigma>\sigma_0\}$,

(ii) $\varphi(\sigma+it)=O((|t|+1)^C)$ for any $\sigma\geq\sigma_0$, with a constant $C>0$,

(iii) It holds that
\begin{align}\label{3-2}
\int_{-T}^T |\varphi(\sigma_0+it)|^2 dt =O(T).
\end{align}

We denote by $\mathcal{M}$ the set of all $\varphi$ satisfying the above conditions.

\begin{remark}\label{automMat}
Here we note that $L_f(s)$ defined in the preceding section belongs to $\mathcal{M}$.
In fact, the Euler product is given by \eqref{2-3}.    The condition \eqref{3-1} is
satisfied with $\alpha=\beta=0$ by \eqref{2-4}, \eqref{2-5}.
It is entire, so (i) is obvious.
Since it satisfies a functional equation, (ii) follows by using the
Phragm{\'e}n-Lindel{\"o}f convexity principle.
Lastly, (iii) follows (with any $\sigma_0>1/2$) by Potter's result \cite{Pot40}. 
\end{remark}

Now let us define $\log\varphi(s)$.   First, when $\sigma>\alpha+\beta+1$, it is defined 
by the sum
$$
\log\varphi(s)=-\sum_{n=1}^{\infty}\sum_{j=1}^{g(n)}{\rm Log}(1-a_n^{(j)}p_n^{-f(j,n)s}),
$$
where Log means the principal branch.    Next, let
$$
B(\rho)=\{\sigma+i\Im\rho\;|\; \sigma_0\leq\sigma\leq\Re\rho\}
$$
for any zero or pole $\rho$ with $\Re\rho\geq\sigma_0$.
We exclude all $B(\rho)$ from $\{s\;|\;\sigma\geq\sigma_0\}$, and denote the remaining set by
$G(\varphi)$.    Then, for any $s\in G(\varphi)$, we may define $\log\varphi(s)$ by the 
analytic continuation along the horizontal path from the right.   Define
$$
V_{\sigma}(T,R;\varphi)=\mu_1\{t\in[-T,T]\;|\;\sigma+it\in G(\varphi),
\log\varphi(\sigma+it)\in R\}.
$$
Then, as a generalization of \eqref{1-2}, the first author \cite{Mat90} proved the following
\begin{theorem}\label{limittheorem}
{\rm (\cite{Mat90})}
Let $\varphi\in\mathcal{M}$.   For any $\sigma>\sigma_0$, the limit
\begin{align}\label{3-3}
W_{\sigma}(R;\varphi)=\lim_{T\to\infty}\frac{1}{2T}V_{\sigma}(T,R;\varphi)
\end{align}
exists.
\end{theorem}

This theorem may be regarded as a result on weak convergence of probability measures, and
Prokhorov's theorem in probability theory is used in the proof given in \cite{Mat90}.

In \cite{Mat92}, the first author presented an alternative argument of proving such a limit
theorem, again without using any convexity.    This argument is based on L{\'e}vy's convergence  theorem.     The method in \cite{Mat92} is more suitable to discuss the matter of density 
functions, so in the present paper we follow the method in \cite{Mat92}. 

In \cite{Mat92}, only the case of Dedekind zeta-functions is discussed, but,
as mentioned in \cite{Mat92b}, the idea in \cite{Mat92} can be applied to any
$\varphi\in\mathcal{M}$.    Such a generalization has, however, not yet been published, so
we will give a sketch of the argument in the following Sections \ref{sec4} and \ref{sec5}. 

%%%%%%%%%%%%%%%%%%%%%%%%%%%%%%%%%%%%%%%%%%%%%%%%%%%%%%%%%%%%%%%%%%%%%%%%%%%%%%%%%%%%%%%%%%%
\section{The method of Fourier transforms}\label{sec4}
%%%%%%%%%%%%%%%%%%%%%%%%%%%%%%%%%%%%%%%%%%%%%%%%%%%%%%%%%%%%%%%%%%%%%%%%%%%%%%%%%%%%%%%%%%%

Let $\sigma>\sigma_0$, and $N\in\mathbb{N}$.    The starting point of the argument is to 
consider the finite trancation of $\varphi(s)$, that is
\begin{align*}
\varphi_N(s)=\prod_{n\leq N}A_n(p_n^{-s})^{-1}
=\prod_{n\leq N}\prod_{j=1}^{g(n)}\left(1-r_n^{(j)}p_n^{-if(j,n)t}\right)^{-1},
\end{align*}
where $r_n^{(j)}=a_n^{(j)}p_n^{-f(j,n)\sigma}$.    Then
\begin{align}\label{4-1}
\log\varphi_N(s)=-\sum_{n\leq N}\sum_{j=1}^{g(n)}\log\left(1-r_n^{(j)}
e^{-itf(j,n)\log p_n}\right).
\end{align}
Note that 
$$
|r_n^{(j)}|\leq |a_n^{(j)}|p_n^{-f(j,n)\sigma}\leq p_n^{\beta-\sigma}
\leq p_n^{\beta-(\alpha+\beta+1/2)}\leq p_n^{-1/2}\leq 1/\sqrt{2}.
$$

Let $\mathbb{Z}$ be the set of all integers, $\mathbb{R}$ the set of all real numbers,
$\mathbb{T}^N=(\mathbb{R}/\mathbb{Z})^N$ be the $N$-dimensional unit torus, and define 
the mapping
$S_N:\mathbb{T}^N\to\mathbb{C}$, attached to \eqref{4-1}, by
\begin{align}\label{4-2}
S_N(\theta_1,\ldots,\theta_N)=-\sum_{n\leq N}\sum_{j=1}^{g(n)}\log\left(1-r_n^{(j)}
e^{2\pi if(j,n)\theta_n}\right).
\end{align}
(Though $S_N$ depends on $\sigma$ and $\varphi$, we do not write explicitly in the
notation, for brevity.    Similar abbreviation is applied to the notation of
$\lambda_N$, $\Lambda$, $K_n$ below.)
We write $z_n^{(j)}(\theta_n)=-\log(1-r_n^{(j)}e^{2\pi if(j,n)\theta_n})$ and
$z_n(\theta_n)=\sum_{j=1}^{g(n)}z_n^{(j)}(\theta_n)$.
Then 
\begin{align}\label{4-3}
S_N(\theta_1,\ldots,\theta_N)=\sum_{n\leq N}z_n(\theta_n).
\end{align}

For any Borel subset $A\subset\mathbb{C}$, we define
$W_{N,\sigma}(A;\varphi)=\mu_N(S_N^{-1}(A))$.    Then $W_{N,\sigma}$ is a probability measure on
$\mathbb{C}$.

Let $R\subset\mathbb{C}$ be any rectangle with the edges parallel to the axes.
The idea of considering the inverse image $S_N^{-1}(R)\subset\mathbb{T}^N$
goes back to Bohr's work (Bohr and Courant \cite{BC14}, Bohr \cite{Boh15}, and Bohr and 
Jessen \cite{BJ3032}).    Also let $E$ be any strip, parallel to the real or imaginary axis.
We have the following two facts, whose proofs of these two facts are exactly the same as 
the proofs of \cite[Lemma 1]{Mat92}.

{\bf Fact 1}.  {\it The sets $S_N^{-1}(R)$, $S_N^{-1}(E)$ are Jordan measurable.}

{\bf Fact 2}.  {\it For any $\varepsilon>0$, there exists a positive number 
$\eta$ such that, 
for any strip $E$ whose width is not larger than $\eta$, it holds that 
$W_{N,\sigma}(E;\varphi)<\varepsilon$.}

Now define
$$
V_{N,\sigma}(T,R;\varphi)=\mu_1\{t\in[-T,T]\;|\;\log\varphi_N(\sigma+it)\in R\}.
$$
We see that $\log\varphi_N(\sigma+it)\in R$ if and only if
$$
\left(\left\{-\frac{t}{2\pi}\log p_1\right\},\ldots,
\left\{-\frac{t}{2\pi}\log p_N\right\}\right)\in S_N^{-1}(R)
$$
(where $\{x\}$ means the fractional part of $x$).
Since $\log p_1,\ldots,\log p_N$ are linearly independent over the rational number field
$\mathbb{Q}$, in view of Fact 1, we can apply the Kronecker-Weyl theorem to obtain
\begin{prop}\label{N-limittheorem}
For any $N\in\mathbb{N}$, we have
\begin{align}\label{4-4}
W_{N,\sigma}(R;\varphi)=\lim_{T\to\infty}\frac{1}{2T}V_{N,\sigma}(T,R;\varphi).
\end{align}
\end{prop}
This is the "finite truncation" version of 
Theorem \ref{limittheorem}.    Therefore, the remaining task to arrive at Theorem
\ref{limittheorem} is to discuss the limit $N\to\infty$.    For this purpose, we consider the
Fourier transform
$$
\Lambda_{N}(w)=\int_{\mathbb{C}}e^{i\langle z,w\rangle}dW_{N,\sigma}(z;\varphi),
$$
where $\langle z,w\rangle=\Re z \Re w+\Im z \Im w$.    Our next aim is to show the following

\begin{prop}\label{limitLambda}
As $N\to\infty$, $\Lambda_{N}(w)$ converges to a certain function 
$\Lambda(w)$, uniformly in
$\{w\in\mathbb{C}\;|\;|w|\leq a\}$ for any $a>0$.
\end{prop}

\begin{proof}
The proof is quite similar to the argument in \cite[Section 3]{Mat92}.
It is easy to see that
$$
\Lambda_{N}(w)=\int_{\mathbb{T}^N}e^{i\langle S_N(\theta_1,\ldots,\theta_N),w\rangle}
d\mu_N(\theta_1,\ldots,\theta_N),
$$
so in view of \eqref{4-3} we can write
\begin{align}\label{4-5}
\Lambda_N(w)=\prod_{n\leq N}K_n(w)
\end{align}
with
$$
K_n(w)=\int_0^1 e^{i\langle z_n(\theta_n),w\rangle} d\theta_n.
$$
%Using the mean value theorem for harmonic functions we have
%$$
%\int_0^1 \Re(z_n^{(j)}(\theta_n))d\theta_n=\int_0^1 \Im(z_n^{(j)}(\theta_n))d\theta_n=0,
%$$
%and hence
%$$
%\int_0^1 \langle z_n(\theta_n),w\rangle d\theta_n=0.
%$$
%Therefore, noting that $|e^{it}-(1+it)|\ll t^2$ holds for any real $t$, we see that
%\begin{align}\label{4-6}
%&|K_n(w)-1|=\left|\int_0^1\left(e^{i\langle z_n(\theta_n),w\rangle}-1-\langle z_n(\theta_n),w
%\rangle\right)d\theta_n\right|\\
%&\ll \int_0^1\langle z_n(\theta_n),w\rangle^2 d\theta_n
%\leq \int_0^1 |z_n(\theta_n)|^2 |w|^2 d\theta_n.\notag
%\end{align}
Noting $|z_n^{(j)}(\theta_n)|\ll |r_n^{(j)}|\leq p_n^{\beta-\sigma}$
and \eqref{3-1}, we have
$$
|z_n(\theta_n)|^2=\left|\sum_{j=1}^{g(n)}z_n^{(j)}(\theta_n)\right|^2
\ll p_n^{2(\alpha+\beta-\sigma)}.
$$
Therefore, analogously to \cite[(3.2)]{Mat92}, we obtain
\begin{align}\label{4-7}
|K_n(w)-1|\ll |w|^2 p_n^{2(\alpha+\beta-\sigma)},
\end{align}
which implies
\begin{align}\label{4-8}
|\Lambda_{n+1}(w)-\Lambda_n(w)|=|\Lambda_n(w)|\cdot|K_{n+1}(w)-1|
\ll |w|^2 p_{n+1}^{2(\alpha+\beta-\sigma)}.
\end{align}
Therefore, for $M> N$, 
\begin{align}\label{4-9}
&|\Lambda_{M}(w)-\Lambda_N(w)|\leq \sum_{n=N}^{M-1}|\Lambda_{n+1}(w)-\Lambda_n(w)|\\
&\ll |w|^2 \sum_{n=N}^{M-1}p_{n+1}^{2(\alpha+\beta-\sigma)}
\leq |w|^2 \sum_{n=N}^{\infty}p_{n+1}^{2(\alpha+\beta-\sigma)}.\notag
\end{align}
Since $\sigma>\sigma_0\geq\alpha+\beta+1/2$, the last sum tends to 0 as $N\to\infty$, 
uniformly in the region $|w|\leq a$.    This implies the assertion of the proposition.
\end{proof}

From Proposition \ref{limitLambda}, in view of L{\'e}vy's convergence theorem, we
immediately obtain

\begin{cor}\label{limitW}
There exists a regular probability measure $W_{\sigma}(\cdot\;;\varphi)$, to which 
$W_{N,\sigma}(\cdot\;;\varphi)$ converges weakly as $N\to\infty$, and
\begin{align}\label{4-10}
\Lambda(w)=\int_{\mathbb{C}}e^{i\langle z,w\rangle}dW_{\sigma}(z;\varphi).
\end{align}
\end{cor}

Moreover, taking the limit $M\to\infty$ on \eqref{4-9}, we obtain
\begin{align}\label{4-11}
|\Lambda(w)-\Lambda_N(w)|\ll |w|^2 \sum_{n=N}^{\infty}p_{n+1}^{2(\alpha+\beta-\sigma)}.
\end{align}

%%%%%%%%%%%%%%%%%%%%%%%%%%%%%%%%%%%%%%%%%%%%%%%%%%%%%%%%%%%%%%%%%%%%%%%%%%%%%%%%%%%%%%%%%%
\section{Proof of Theorem \ref{limittheorem}}\label{sec5}
%%%%%%%%%%%%%%%%%%%%%%%%%%%%%%%%%%%%%%%%%%%%%%%%%%%%%%%%%%%%%%%%%%%%%%%%%%%%%%%%%%%%%%%%%%%

In this section we show how to prove Theorem \ref{limittheorem} in the framework of our
present method.    The argument is very similar to that given in \cite[Sections 3 and 4]{Mat92},
so we omit some details.

First, using Fact 2 in Section \ref{sec4}, we can show 
(analogously to the argument in the last part of \cite[Section 3]{Mat92})
that $R$ is a continuity set
with respect to $W_{\sigma}$, and hence
\begin{align}\label{5-1}
W_{\sigma}(R;\varphi)=\lim_{N\to\infty}W_{N,\sigma}(R;\varphi).
\end{align}

Now, following the method in \cite[Section 4]{Mat92}, we prove Theorem
\ref{limittheorem}.     Put
$$
R_N(s;\varphi)=\log\varphi(s)-\log\varphi_N(s),\quad
f_N(s;\varphi)=\frac{\varphi(s)}{\varphi_N(s)}-1.
$$

When $\sigma>\alpha+\beta+1$, since 
\begin{align}\label{5-2}
R_N(s;\varphi)\ll \sum_{n>N}\sum_{j=1}^{g(n)}|a_n^{(j)}|p_n^{-f(j,n)\sigma}
\ll \sum_{n>N}p_n^{\alpha+\beta-\sigma}
\end{align}
which tends to 0 as $N\to\infty$, 
the assertion of the theorem directly follows from Proposition
\ref{N-limittheorem} and \eqref{5-1}.

In the case $\sigma_0<\sigma\leq \alpha+\beta+1$, naturally we have to discuss more
carefully.   Let $\delta>0$, and define
\begin{align*}
K_N^{\delta}(T;\varphi)=\left\{t\in[-T,T]\;\left|\;
\begin{array}{ll}
\sigma+it\in G(\varphi),\\
|\log\varphi(\sigma+it)-\log\varphi_N(\sigma+it)|\geq\delta
\end{array}
\right.\right\},
\end{align*}
and $k_N^{\delta}(T;\varphi)=\mu_1(K_N^{\delta}(T;\varphi))$.
We will prove that $k_N^{\delta}(T;\varphi)$ is negligible, that is, for any $\varepsilon>0$
we can choose $N_0=N_0(\delta,\varepsilon)$ for which
\begin{align}\label{5-3}
\limsup_{T\to\infty}T^{-1}k_N^{\delta}(T;\varphi)\leq\varepsilon
\end{align}
holds for any $N\geq N_0$.

Let $\alpha_0=\sigma-\varepsilon$, $\alpha_1=\sigma-2\varepsilon$.    We choose
$\varepsilon$ so small that $\sigma_0<\alpha_1<\alpha_0<\sigma$.
For any $t_0\in[-T,T]$, put
\begin{align*}
H(t_0)=\{s\;|\;\sigma>\alpha_0,t_0-1/2<t<t_0+1/2\},
\end{align*}
and define $\psi_N^{\delta}(t_0;\varphi)=0$ if $H(t_0)\subset G(\varphi)$ and
$|R_N(s;\varphi)|<\delta$ for any $s\in H(t_0)$, and $\psi_N^{\delta}(t_0;\varphi)=1$
otherwise.    Then clearly
\begin{align}\label{5-4}
k_N^{\delta}(T;\varphi)\leq \int_{-T}^T \psi_N^{\delta}(t_0;\varphi)dt_0.
\end{align}
Using \eqref{5-2} we can find $\beta_0=\alpha+\beta+1+C\delta^{-1}$ (with an absolute
positive constant $C$) for which $|R_N(s;\varphi)|<\delta$ holds for any $s$ satisfying
$\sigma\geq \beta_0$.
Let $Q(t_0)=H(t_0)\cap \{s\;|\;\sigma<\beta_0\}$.

\begin{lemma}\label{bohrlemma}
If $|f_N(s;\varphi)|<\delta/2$ for any $s\in Q(t_0)$, then 
$\psi_N^{\delta}(t_0;\varphi)=0$.
\end{lemma}

This is a generalization of \cite[Lemma 2]{Mat92}, which further goes back to Bohr
\cite[Hilfssatz 5]{Boh15}.    Bohr's proof in \cite{Boh15} can be applied without
change to the above general case, so we omit the proof.

Let $\beta_1=2\beta_0$, and let $P(t_0)$ be the rectangle given by
$\alpha_1\leq\sigma\leq\beta_1$, $t_0-1\leq t\leq t_0+1$.    Put
$$
F_N(t_0;\varphi)=\int\!\!\!\int_{P(t_0)} |f_N(s;\varphi)|^2 d\sigma dt.
$$
(This can be defined only when $P(t_0)$ does not include a pole of $\varphi(s)$.)
We use Lemma \ref{bohrlemma} and \cite[Lemma 3]{Mat92} 
(which is \cite[Hilfssatz 4]{Boh15}) to see that if
$$
F_N(t_0;\varphi) < \pi\left(\varepsilon/2\right)^2
\left(\delta/2\right)^2
$$
then $\psi_N^{\delta}(t_0;\varphi)=0$.    Therefore
\begin{align}\label{5-5}
\frac{1}{2T}\int_{-T}^T \psi_N^{\delta}(t_0;\varphi)dt_0\leq b+
\frac{\mu_1(\mathcal{S})}{2T},
\end{align}
where $\mathcal{S}$ is the set of all $t\in[-T,T]$ for which we can find a pole $s'$ of
$\varphi(s)$ satisfying $|t-\Im s'|\leq 2$, and
$$
b=\frac{1}{2T}\mu_1\biggl(\biggl\{t_0\in [-T,T]\setminus\mathcal{S}\; \biggl| \;
F_N(t_0;\varphi)\geq \pi(\varepsilon/2)^2(\delta/2)^2\biggr.\biggr\}\biggr).
$$
From the definition of $b$ we obtain
\begin{align*}
&\pi(\varepsilon/2)^2(\delta/2)^2 b
\leq \frac{1}{2T}\int_{t_0\in [-T,T]\setminus\mathcal{S}} F_N(t_0;\varphi)dt_0\\
&\quad=\frac{1}{2T}\int_{\alpha_1}^{\beta_1}\int_{-T-1}^{T+1}|f_N(s;\varphi)|^2\int^{\#}
 dt_0 dt d\sigma,
\end{align*}
where the innermost integral (with the $\#$ symbol) is on 
$t_0\in[-T,T]\setminus\mathcal{S}$, $t-1\leq t_0\leq t+1$.
This innermost integral is trivially $\leq 2$, and is equal to 0 if there exists a pole
$s'$ of $\varphi(s)$ such that $|t-\Im s'|\leq 1$ (because then all $t_0\in[t-1,t+1]$
belongs to $\mathcal{S}$).    Therefore
\begin{align}\label{5-6}
\pi(\varepsilon/2)^2(\delta/2)^2 b\leq \frac{1}{T}\int_{\alpha_1}^{\beta_1}
\int_{J(T+1)} |f_N(s;\varphi)|^2 dtd\sigma,
\end{align}
where 
$$
J(T)=\{t\in [-T,T]\;|\;|t-\Im s'|>1 \;\mbox{for any pole} \;s'\; \mbox{of}\; \varphi(s)\}.
$$
From \eqref{5-4}, \eqref{5-5} and \eqref{5-6} we now obtain
\begin{align}\label{5-7}
\frac{1}{2T}k_N^{\delta}(T;\varphi)\leq \frac{1}{\pi(\varepsilon/2)^2(\delta/2)^2 T}
\int_{\alpha_1}^{\beta_1}
\int_{J(T+1)} |f_N(s;\varphi)|^2 dtd\sigma + \frac{\mu_1(\mathcal{S})}{2T}.
\end{align}

On the double integral on the right-hand side, as an analogue of \cite[Lemma 4]{Mat92}, we can show the following lemma.

\begin{lemma}\label{carlson}
For any $\eta>0$, There exists $N_0=N_0(\eta)$, such that
\begin{align}\label{5-8}
\frac{1}{T}\int_{\alpha_1}^{\beta_1}
\int_{J(T+1)} |f_N(s;\varphi)|^2 dtd\sigma < \eta
\end{align}
for any $N\geq N_0$ and any $T\geq T_0$ with some $T_0=T_0(N)$.
\end{lemma}

\begin{proof}
Write the Dirichlet series expansion of $\varphi(s)$ in the region 
$\sigma>\alpha+\beta+1$ as
$$
\varphi(s)=\sum_{k=1}^{\infty}c_k k^{-s}.
$$
Then the Dirichlet series expansion of $f_N(s)$ is
\begin{align*}
f_N(s;\varphi)={\sum_{k}}^{\prime}c_k k^{-s},
\end{align*}
where the symbol $\sum^{\prime}$ means that the summation is restricted to $k>1$ which is
co-prime with $p_1p_2\cdots p_N$.    In \cite[Appendix]{KM17} it has been shown that,
for any $\varepsilon>0$, we can choose a sufficiently large $N=N(\varepsilon)$ such that
\begin{align}\label{5-9}
c_k=O(k^{\alpha+\beta+\varepsilon})
\end{align}
for all $k$ co-prime with $p_1p_2\cdots p_N$. 

By \eqref{3-2} and the convexity principle we have
\begin{align}\label{5-10}
\int_{J(T)}|\varphi(\sigma+it)|^2 dt=O(T)
\end{align}
for any $\sigma\geq \sigma_0$.   On the other hand, using \eqref{4-1} we have
\begin{align*}
\varphi_N(\sigma+it)^{-1}
\leq\exp\left(C\sum_{n\leq N}\sum_{j=1}^{g(n)}|a_n^{(j)}|p_n^{-f(j,n)\sigma}\right)
\leq\exp\left(C' N^{\alpha+\beta+1-\sigma}\right)
\end{align*}
(where $C,C'$ are positive constants).    Combining this estimate with \eqref{5-10} 
we obtain
$$
\frac{1}{T}\int_{J(T)}|f_N(\sigma+it;\varphi)|^2 dt \ll 
\exp\left(2C' N^{\alpha+\beta+1-\sigma}\right),
$$
which is $O(1)$ with respect to $T$.    Therefore by Carlson's mean value theorem
(see \cite[Section 9.51]{Tit39})
\begin{align}\label{5-11}
\lim_{T\to\infty}\frac{1}{T}\int_{J(T)}|f_N(\sigma+it;\varphi)|^2 dt 
= {\sum_{k}}^{\prime}c_k^2 k^{-2\sigma},
\end{align}
uniformly in $\sigma$.    Using \eqref{5-9}, we can estimate the right-hand side of
\eqref{5-11} as
$$
\ll \sum_{k\geq p_{N+1}}k^{2(\alpha+\beta+\varepsilon-\sigma)}
\ll N^{1+2(\alpha+\beta+\varepsilon-\sigma)},
$$
whose exponent is negative for $\sigma>\sigma_0$ (if $\varepsilon$ is sufficiently small).
This immediately implies the assertion of the lemma.
\end{proof}

Now, applying Lemma \ref{carlson} with $\eta=\pi\delta^2\varepsilon^3/16$ to
\eqref{5-7}, we arrive at \eqref{5-3}.
The assertion of the theorem in the case $\sigma_0<\sigma\leq\alpha+\beta+1$ then follows
by the same argument as in the last part of \cite[Section 4]{Mat92}.

%%%%%%%%%%%%%%%%%%%%%%%%%%%%%%%%%%%%%%%%%%%%%%%%%%%%%%%%%%%%%%%%%%%%%%%%%%%%%%%%%%%%%%%%%%%%
\section{The density function}\label{sec6}
%%%%%%%%%%%%%%%%%%%%%%%%%%%%%%%%%%%%%%%%%%%%%%%%%%%%%%%%%%%%%%%%%%%%%%%%%%%%%%%%%%%%%%%%%%%%

In this section $\sigma$ is any real number larger than $\sigma_0$.
We discuss when it is possible to show that
$W_{\sigma}(\cdot;L_f)$ is absolutely continuous.    Then by measure theory we can write
\begin{align}\label{6-1}
W_{\sigma}(R;\varphi)=\int_R \mathcal{M}_{\sigma}(z,\varphi)|dz|
\end{align}
with the Radon-Nikod{\'y}m density function $\mathcal{M}_{\sigma}(z;\varphi)$.

For this purpose, we aim to show
\begin{align}\label{6-2}
\Lambda_N(w)=O(|w|^{-(2+\eta)}) \qquad (|w|\to\infty)
\end{align}
uniformly in $N$, with some $\eta>0$.    

If \eqref{6-2} is valid, then 
$$
\int_{\mathbb{C}}|\Lambda_N(w)||dw|<\infty.
$$
Therefore $W_{N,\sigma}$ is absolutely continuous, and the
Radon-Nikod{\'y}m density function $\mathcal{M}_{N,\sigma}(z;\varphi)$ is given by
\begin{align}\label{6-3}
\mathcal{M}_{N,\sigma}(z;\varphi)=\int_{\mathbb{C}}
e^{-i\langle z,w\rangle}\Lambda_N(w)|dw|
\end{align} 
and is continuous
(see \cite[p.53]{JW35}, \cite[p.105]{BJ48}).
Moreover, the above uniformity in $N$ implies that the same estimate as \eqref{6-2}
is valid for the limit function $\Lambda(w)$.    Therefore $W_{\sigma}$ is also
absolutely continuous, hence \eqref{6-1} is valid with
the continuous density function given by
\begin{align}\label{6-4}
\mathcal{M}_{\sigma}(z;\varphi)=\int_{\mathbb{C}}
e^{-i\langle z,w\rangle}\Lambda(w)|dw|.
\end{align} 

The following proposition reduces the problem to the evaluation of $K_n(w)$:

\begin{prop}\label{at-least-5-n}
If there are at least five $n$'s, say $n_1,\ldots,n_5$, for which
$K_n(w)=O_n(|w|^{-1/2})$ holds as $|w|\to\infty$, then \eqref{6-2} is valid
for any $N\geq\max\{n_1,\ldots,n_5\}$, and so \eqref{6-1} and
\eqref{6-4} are also valid.
\end{prop}

\begin{remark}
The proof of \eqref{6-2} in the above proposition is simple: just apply
$K_n(w)=O_n(|w|^{-1/2})$ (for $n_1,\ldots,n_5$) and the trivial estimate
$|K_n(w)|\leq 1$ to the product formula \eqref{4-5}.    
The result is \eqref{6-2} with $\eta=1/2$, uniform in $N$. 
\end{remark}

\begin{remark}\label{rem-quantitative}
The existence of the density function is useful for quantitative studies.
For instance, if there are at least ten $n$'s with $K_n(w)=O(|w|^{-1/2})$, then 
$\Lambda_N(w)=O(|w|^{-5})$ for large $N$.    This fact with \eqref{4-7}, \eqref{4-11}
leads the estimate
\begin{align}\label{6-5}
|W_{\sigma}(R;L_f)-W_{N,\sigma}(R;L_f)|=O(\mu_2(R)N^{1+2(\alpha+\beta-\sigma)}
(\log N)^{2(\alpha+\beta-\sigma)})
\end{align}
for $\sigma>\sigma_0$, as an analogue of \cite[(6.4)]{Mat92}.
\end{remark}

In \cite{Mat92}, when $\varphi=\zeta_K$ (the Dedekind zeta-function of a Galois number
field $K$), the key estimate \eqref{6-2} was proved by using \cite[Theorem 13]{JW35}.     
In this case,
$\zeta_K$ has the Euler product of the form \eqref{3-0} with 
$f(1,n)=\cdots=f(g(n),n)$ ($=f(n)$, say, the inertia degree) and $a_n^{(j)}=1$ (and hence
$r_n^{(1)}=\cdots=r_n^{(g(n))}=p_n^{-f(n)\sigma}$ ($=r_n$, say)).   Therefore
$$
z_n(\theta_n)=-g(n)\log(1-r_n e^{2\pi if(n)\theta_n}),
$$
which describes a curve when $\theta_n$ moves from 0 to 1.    This curve is
convex, so the original Jessen-Wintner inequality (\cite[Theorem 13]{JW35}) can be
directly applied.    In this case we encounter only one type of curve, that is, the curve
$-\log(1-\xi)$ ($\xi\in\mathbb{C}$, $|\xi|=r_n$).

When $K$ is non-Galois, $f(1,n),\ldots,f(g(n),n)$ are not necessarily the same as each
other, so 
$$
z_n(\theta_n)=-\sum_{j=1}^{g(n)}\log(1-r_n^{(j)}e^{2\pi if(j,n)\theta_n}).
$$
However, still in this case, the number of relevant types of curves 
$$
-\sum_{j=1}^{g(n)}\log(1-\xi^{f(j,n)}) \qquad(\xi\in\mathbb{C},\; |\xi|=p_n^{-\sigma})
$$ 
is finite, because there are only finitely many patterns of the decomposition of prime
numbers into prime ideals in $K$. 
Because of this finiteness, we can use \cite[Lemma 2]{Mat07} (which is a simple
generalization of \cite[Theorem 13]{JW35}) to show \eqref{6-2} in this case.
The case of Hecke $L$-functions of ideal class characters can be treated in a
similar way.

However in the automorphic case, we encounter infinitely many types of curves,
because in this case $z_n(\theta_n)$ describes a curve
\begin{align}\label{6-6}
-\log(1-\alpha_f(p_n)\xi)-\log(1-\beta_f(p_n)\xi)  \qquad (\xi\in\mathbb{C},\;
|\xi|=p_n^{-\sigma}),
\end{align}
which depends on $\alpha_f(p_n),\beta_f(p_n)$.
Therefore we have to prove a new type of Jessen-Wintner inequality, suitable for
the automorphic case.    This will be done in the next section.

%%%%%%%%%%%%%%%%%%%%%%%%%%%%%%%%%%%%%%%%%%%%%%%%%%%%%%%%%%%%%%%%%%%%%%%%%%%%%%%%%%%%%%%%%%%%
\section{An analogue of the Jessen-Wintner inequality for automorphic 
$L$-functions}\label{sec7}
%%%%%%%%%%%%%%%%%%%%%%%%%%%%%%%%%%%%%%%%%%%%%%%%%%%%%%%%%%%%%%%%%%%%%%%%%%%%%%%%%%%%%%%%%%%%

Now we restrict ourselves to the case of automorphic $L$-functions.    Except for 
the (finitely many) prime factors of $N$, the Euler factor of $L_f(s)$ is of the form
$$
(1-\alpha_f(p_n)p_n^{-s})^{-1}(1-\beta_f(p_n)p_n^{-s})^{-1},
$$
so $z_n(\theta_n)=A_n(p_n^{-\sigma}e^{2\pi i\theta_n})$ with
$$
A_n(X)=-\log(1-\alpha_f(p_n)X)-\log(1-\beta_f(p_n)X).
$$
When $\theta_n$ moves from 0 to 1, the points $z_n(\theta_n)$ describes a curve 
\eqref{6-6} on the
complex plane, which we denote by $\Gamma_n$.

Let $x_{n}(\theta_n)=\Re z_n(\theta_n)$ and $y_{n}(\theta_n)=\Im z_n(\theta_n)$.
Writing $w=|w|e^{i\tau}$ ($\tau\in[0,2\pi)$) we have
$w=|w|\cos\tau+i|w|\sin\tau$.   Then
\begin{align}\label{7-1}
\langle z_n(\theta_n),w\rangle=|w| g_{\tau,n}(\theta_n),
\end{align}
where
\begin{align*}
g_{\tau,n}(\theta_n)=x_{n}(\theta_n)\cos\tau+y_{n}(\theta_n)\sin\tau.
\end{align*}
Therefore
\begin{align}\label{7-2}
K_n(w)=\int_0^1 e^{i|w|g_{\tau,n}(\theta_n)}d\theta_n.
\end{align}

\begin{lemma}\label{lemma-curve}
Let $n\in\mathbb{N}$ such that $p_n\nmid N$.
For any fixed $\tau$, the function $g_{\tau,n}(\theta_n)$ (as a function in $\theta_n$) is
a $C^{\infty}$-class function.    Moreover, if $p_n\in\mathbb{P}_f(\varepsilon)$ and
$n$ is sufficiently large, then 
$g_{\tau,n}^{\prime\prime}(\theta_n)$ has exactly two zeros 
on the interval $[0,1)$.
\end{lemma}

\begin{proof}
Hereafter, for brevity, we write $p_n=p$, $p_n^{-\sigma}=q$, $2\pi\theta_n=\theta$, 
$z_n(\theta_n)=z(\theta)$, $g_{\tau,n}(\theta_n)=g_{\tau}(\theta)$, 
$x_{n}(\theta_n)=x(\theta)$, and $y_n(\theta_n)=y(\theta)$.   
Since the
Taylor expansion of $A_n(x)$ is given by
$$
A_n(x)=\sum_{j=1}^{\infty}a_j x^j\qquad\mbox{with}\qquad
a_j=\frac{1}{j}(\alpha_f(p)^j+\beta_f(p)^j),
$$
we have
$$
z(\theta)=\sum_{j=1}^{\infty}a_j q^j e^{ij\theta}.
$$
Therefore, putting $b_j=\Re a_j$ and $c_j=\Im a_j$, we have
\begin{align*}
x(\theta)=\sum_{j=1}^{\infty}q^j u_j(\theta),\quad
y(\theta)=\sum_{j=1}^{\infty}q^j v_j(\theta),
\end{align*}
where
$$
u_j(\theta)=b_j\cos(j\theta)-c_j\sin(j\theta), \quad
v_j(\theta)=b_j\sin(j\theta)+c_j\cos(j\theta).
$$
Differentiate these series termwisely with respect to $\theta$; for example
\begin{align*}
x^{\prime}(\theta)=-\sum_{j=1}^{\infty}jq^j v_j(\theta),\quad
y^{\prime}(\theta)=\sum_{j=1}^{\infty}jq^j u_j(\theta)
\end{align*}
and so on.
From \eqref{2-4} we have  $|a_j|\leq 2/j$, so
\begin{align}\label{estimate-b-c}
|b_j|\leq 2/j,\quad |c_j|\leq 2/j.
\end{align}
Noting these estimates and $q<1$, we see that
these differentiated series are 
convergent absolutely.
Therefore $x(\theta)$, $y(\theta)$ are belonging to the $C^{\infty}$-class, 
and so is $g_{\tau}(\theta)$.
In particular the above termwise differentiation is valid, and we have
\begin{align}\label{7-3}
g_{\tau}^{\prime}(\theta)
&=-\sum_{j=1}^{\infty}jq^j v_j(\theta)\cos\tau
+\sum_{j=1}^{\infty}jq^j u_j(\theta)\sin\tau\\
&=-q v_1(\theta)\cos\tau+q u_1(\theta)\sin\tau
+E_1(q;\theta,\tau),\notag
\end{align}
where $E_1(q;\theta,\tau)$ denotes the sum corresponding to $j\geq 2$, and 
\begin{align}\label{7-4}
|E_1(q;\theta,\tau)|&\leq 2\sum_{j\geq 2}jq^j(|b_j|+|c_j|)
\leq 2\sum_{j\geq 2}jq^j\left(\frac{2}{j}+\frac{2}{j}\right)\\
&=8\sum_{j\geq 2}q^j =\frac{8q^2}{1-q}.\notag
\end{align}
Since $q=p_n^{-\sigma}\leq 2^{-1/2}=1/\sqrt{2}$, we find that
$E_1(q;\theta,\tau)=O(q^2)$ as $q\to 0$ (that is, $n\to\infty$), 
where the implied constant is absolute.
Therefore from \eqref{7-3} we have
\begin{align*}
g_{\tau}^{\prime}(\theta)
=-qb_1\sin(\theta-\tau)-qc_1\cos(\theta-\tau)+O(q^2).
\end{align*}
Write $\gamma_1=\arg a_1$.    Then $b_1=|a_1|\cos\gamma_1$, $c_1=|a_1|\sin\gamma_1$,
and so
\begin{align}\label{7-5}
g_{\tau}^{\prime}(\theta)
&=-q|a_1|(\cos\gamma_1\sin(\theta-\tau)+\sin\gamma_1\cos(\theta-\tau))+O(q^2)\\
&=-q(|a_1|\sin(\gamma_1+\theta-\tau)+O(q)).\notag
\end{align}
Similarly, one more differentiation gives
\begin{align}\label{7-6}
g_{\tau}^{\prime\prime}(\theta)
&=-\sum_{j=1}^{\infty}j^2q^j u_j(\theta)\cos\tau
-\sum_{j=1}^{\infty}j^2q^j v_j(\theta)\sin\tau\\
&=-q|a_1|\cos(\gamma_1+\theta-\tau)
+E_2(q;\theta,\tau),\notag
\end{align}
where $E_2(q;\theta,\tau)$, the sum corresponding to $j\geq 2$, satisfies
\begin{align}\label{7-7}
|E_2(q;\theta,\tau)|&\leq 2\sum_{j\geq 2}j^2q^j(|b_j|+|c_j|)
\leq 2\sum_{j\geq 2}j^2q^j\left(\frac{2}{j}+\frac{2}{j}\right)\\
&=8\sum_{j\geq 2}jq^j =\frac{8q^2(2-q)}{(1-q)^2}.\notag
\end{align}
(The proof of the last equality: Put $J=\sum_{j\geq 2}jq^j$, and observe that
$J=\sum_{j\geq 1}(j+1)q^{j+1}=q\sum_{j\geq 1}jq^{j}+\sum_{j\geq 1}q^{j+1}
=q^2+qJ+q^2/(1-q)$.)
Therefore $E_2(q;\theta,\tau)=O(q^2)$ with an absolute implied constant (by using again
$q\leq 1/\sqrt{2}$),
and hence
\begin{align}\label{7-8}
g_{\tau}^{\prime\prime}(\theta)=-q(|a_1|\cos(\gamma_1+\theta-\tau)+O(q)).
\end{align}
Furthermore
\begin{align}\label{7-9}
g_{\tau}^{\prime\prime\prime}(\theta)
=q|a_1|\sin(\gamma_1+\theta-\tau)+E_3(q;\theta,\tau)
\end{align}
with
\begin{align}\label{7-10}
&|E_3(q;\theta,\tau)|\leq 2\sum_{j\geq 2}j^3 q^j (|b_j|+|c_j|)
\leq 8\sum_{j\geq 2}j^2 q^j\\
&\qquad=8q^2\left(\frac{3}{1-q}+\frac{1}{1-q^2}+\frac{2q(2-q)}{(1-q)^3}\right)
=O(q^2)\notag
\end{align}
with an absolute implied constant.
(The evaluation of $\sum_{j\geq 2}j^2q^j$ can be done similarly to the last
equality of \eqref{7-7}.)

Now we assume that $p_n\in\mathbb{P}_f(\varepsilon)$, where $\varepsilon$ is a small
positive number.
Recall $a_1=\alpha_f(p)+\beta_f(p)=\lambda_f(p)$.
Therefore from \eqref{2-6} we have
$|a_1|>\sqrt{2}-\varepsilon$.
On the other hand, the term $O(q)$ can be arbitrarily small when $n$ is 
sufficiently large.    Therefore from \eqref{7-8} we find that, for sufficiently 
large $n$, if $\theta=\theta_0$ is 
a solution of $g_{\tau}^{\prime\prime}(\theta)=0$, then 
$\cos(\gamma_1+\theta_0-\tau)$ is to be close to 0.    That is, writing 
$\theta=\theta_1^c,\theta_2^c$ be two solutions of $\cos(\gamma_1+\theta-\tau)=0$ in
the interval $0\leq\theta<2\pi$, we see that $\theta_0$ is close to $\theta_1^c$ or
$\theta_2^c$.

Now consider $g_{\tau}^{\prime\prime\prime}(\theta)$.    From \eqref{7-9} and
\eqref{7-10} we have
$$
g_{\tau}^{\prime\prime\prime}(\theta)=q(|a_1|\sin(\gamma_1+\theta-\tau)+O(q)).
$$
Since
$$
|\sin(\gamma_1+\theta_1^c-\tau)|=|\sin(\gamma_1+\theta_2^c-\tau)|=1,
$$
we see that $g_{\tau}^{\prime\prime\prime}(\theta)\neq 0$ around
$\theta=\theta_j^c$ ($j=1,2$), if $p_n\in\mathbb{P}_f(\varepsilon)$ and $n$
is sufficiently large.    This implies that $g_{\tau}^{\prime\prime}(\theta)$ is
monotone around $\theta=\theta_j^c$.    Therefore we conclude that there is at most
one solution $\theta=\theta_0$ of $g_{\tau}^{\prime\prime}(\theta)=0$ around
$\theta_j^c$.

Moreover, from \eqref{7-8} we see that $g_{\tau}^{\prime\prime}(\theta)$ is
negative around the value of $\theta$ satisfying $\cos(\gamma_1+\theta-\tau)=1$, 
and is positive around the value of $\theta$ satisfying $\cos(\gamma_1+\theta-\tau)=-1$.
Therefore $g_{\tau}^{\prime\prime}(\theta)$ changes its sign twice in the
interval $0\leq\theta<1$, so that the above solution $\theta_0$ indeed exists both around
$\theta_1^c$ and around $\theta_2^c$.    
We denote those solutions by $\theta_1^{\prime\prime}$ and $\theta_2^{\prime\prime}$,
respectively.    That is, $g_{\tau}^{\prime\prime}(\theta)=0$
has exactly two solutions in the interval $0\leq\theta<2\pi$.
\end{proof}

\begin{remark}\label{rem-g-prime}
By the same reasoning as above, we can show that,
if $p_n\in\mathbb{P}_f(\varepsilon)$ and $n$ is sufficiently large, 
$g_{\tau}^{\prime}(\theta)=0$ also
has exactly two solutions $\theta_1^{\prime}$ and $\theta_2^{\prime}$
in the interval $0\leq\theta<2\pi$. 
In fact, there exists two solutions 
$\theta=\theta_1^s,\theta_2^s$ of $\sin(\gamma_1+\theta-\tau)=0$ in the interval $0\leq\theta<2\pi$, and
$\theta_j^{\prime}$ is close to $\theta_j^s$ ($j=1,2$).
(We can further show that, for any $l\in\mathbb{N}$, there exist exactly two solutions
of $g_{\tau}^{(l)}(\theta)=0$.)
\end{remark}

Now we can prove an analogue of the Jessen-Wintner inequality for
automorphic $L$-functions.
In the rest of this section, we follow the argument in the proof of
\cite[Theorem 12]{JW35}.
We use the notation defined in the proof of Lemma \ref{lemma-curve}
and in Remark \ref{rem-g-prime}.  
The integral \eqref{7-2} can be rewritten as
\begin{align}\label{7-11}
K_n(w)=\frac{1}{2\pi}\int_0^{2\pi} e^{i|w|g_{\tau}(\theta)}d\theta.
\end{align}

\begin{prop}\label{prop-curve}
If $p_n\in\mathbb{P}_f(\varepsilon)$ and $n$ is sufficiently large, we have
\begin{align*}
K_n(w)=O\left(\frac{1}{q^{1/2}|w|^{1/2}}+\frac{1}{q|w|}\right),
\end{align*}
with the implied constant depending only on $\varepsilon$.
\end{prop}

\begin{proof}
When $\theta$ moves between $\theta_i^s$ and $\theta_j^c$ ($1\leq i,j\leq 2$)
(mod $2\pi$),
the values of $\sin(\gamma_1+\theta-\tau)$ and $\cos(\gamma_1+\theta-\tau)$
varies continuously and monotonically, and there exists a unique value
$\theta=\theta_{ij}$ between $\theta_i^s$ and $\theta_j^c$ at which
$$|\sin(\gamma_1+\theta_{ij}-\tau)|=|\cos(\gamma_1+\theta_{ij}-\tau)|=1/\sqrt{2}$$ 
holds.

We split the interval $0\leq\theta<1$ (mod $2\pi$) into four subintervals at the values 
$\theta_{ij}$ ($1\leq i,j\leq 2$).
Then on two of those subintervals (which we denote by $I_A$ and $I_B$) the inequality
$|\sin(\gamma_1+\theta-\tau)|\geq 1/\sqrt{2}$ holds, while on the other two
subintervals (which we denote by $I_C$ and $I_D$) 
the inequality $|\cos(\gamma_1+\theta-\tau)|\geq 1/\sqrt{2}$ holds.

Since $p_n\in\mathbb{P}_f(\varepsilon)$ and $n$ is sufficiently large, we can again
use the facts $|a_1|>\sqrt{2}-\varepsilon$ and
the term $O(q)$ is small.
Therefore from \eqref{7-5} we find 
\begin{align}\label{7-12}
|g_{\tau}^{\prime}(\theta)|\geq q((\sqrt{2}-\varepsilon)(1/\sqrt{2})-\varepsilon)
\geq q(1-2\varepsilon)
\end{align}
for $\theta\in I_A\cup I_B$.   
Similarly from \eqref{7-8} we find that, for
sufficiently large $n$,
\begin{align}\label{7-13}
|g_{\tau}^{\prime\prime}(\theta)|\geq q(1-2\varepsilon)
\end{align}
for $\theta\in I_C\cup I_D$.

The number $\theta_1^c$ is included in $I_A$ or $I_B$, say $I_A$.    Then 
$\theta_2^c\in I_B$.
Therefore also $\theta_1^{\prime\prime}\in I_A$ and $\theta_2^{\prime\prime}\in I_B$.
We split $I_A$ into two subintervals at $\theta=\theta_1^{\prime\prime}$.
Then in the both of those subintervals, $g_{\tau}^{\prime}(\theta)$ is monotone.
Therefore, applying the first derivative test (Titchmarsh \cite[Lemma 4.2]{Tit51})
with \eqref{7-12} to those subintervals we have
$$
\left|\int_{I_A}e^{i|w|g_{\tau}(\theta)}d\theta\right|\leq 
2\cdot\frac{4}{\min \{|w||g_{\tau}^{\prime}(\theta)|\}}
\leq \frac{8}{q|w|(1-2\varepsilon)},
$$
and the same inequality holds for the integral on $I_B$.

As for the integrals on the intervals $I_C$ and $I_D$, we use the second derivative test
(\cite[Lemma 4.4]{Tit51}).    The monotonicity is not required for the second derivative
test, so we need not divide $I_C$ into subintervals.
Using \eqref{7-13}, we have
$$
\left|\int_{I_C}e^{i|w|g_{\tau}(\theta)}d\theta\right|\leq 
\frac{8}{\sqrt{|w|q(1-2\varepsilon)}},
$$
and the same for $I_D$.    Collecting these inequalities, we obtain the assertion
of the proposition.
\end{proof}

Proposition \ref{prop-curve} implies that
\begin{align}\label{7-14} 
K_n(w)=O_{n,\varepsilon}(|w|^{-1/2})\qquad (|w|\to\infty)
\end{align}
if $p_n\in\mathbb{P}_f(\varepsilon)$ and $n$ is sufficiently large.
The set $\mathbb{P}_f(\varepsilon)$ is of positive density, especially it
includes infinitely many elements (so surely includes five elements).    
Therefore we can obviously apply
Proposition \ref{at-least-5-n} to $\varphi(s)=L_f(s)$, and the proof of
Theorem \ref{main-1} is now complete.

%%%%%%%%%%%%%%%%%%%%%%%%%%%%%%%%%%%%%%%%%%%%%%%%%%%%%%%%%%%%%%%%%%%%%%%%%%%%%%%
\section{The convexity}\label{sec8}
%%%%%%%%%%%%%%%%%%%%%%%%%%%%%%%%%%%%%%%%%%%%%%%%%%%%%%%%%%%%%%%%%%%%%%%%%%%%%%%
In our proof of Theorem \ref{main-1}, the convexity of relevant curves plays no
role.    However the geometric property of the curve $\Gamma_n$ is of independent
interest.    We conclude this paper with the following

\begin{prop}\label{convexity}
If $p_n\in\mathbb{P}_f(\varepsilon)$ for a small positive number $\varepsilon$ 
and $n$ is sufficiently large, the curve
$\Gamma_n$ is a closed convex curve.
\end{prop}

\begin{remark}\label{rem-convexity}
Using \cite[Theorem 13]{JW35} we have that each curve $\Gamma_n$ is convex if
$|\xi|$ is sufficiently small.    But their theorem does not give any explicit bound
of $|\xi|$ (which may depend on $n$), so we cannot deduce the above proposition from 
their theorem.
\end{remark}

\begin{proof}[Proof of Proposition \ref{convexity}]
Assume $p_n\in\mathbb{P}_f(\varepsilon)$ and $n$ is large.    Then
$$
u_1(\theta)^2+v_1(\theta)^2=b_1^2+c_1^2=|a_1|^2
=|\alpha_f(p)+\beta_f(p)|^2>(\sqrt{2}-\varepsilon)^2
$$
by \eqref{2-6}.
Therefore at least one of $|u_1(\theta)|^2$ and $|v_1(\theta)|^2$ is larger than
$(\sqrt{2}-\varepsilon)^2/2$, that is, at least one of 
$|u_1(\theta)|$ and $|v_1(\theta)|$ is larger than
$(\sqrt{2}-\varepsilon)/\sqrt{2}>1-\varepsilon$.
Let
\begin{align*}
\Theta(u_1,n)&=\{\theta\in[0,2\pi)\;|\;|u_1(\theta)|>1-\varepsilon\},\\
\Theta(v_1,n)&=\{\theta\in[0,2\pi)\;|\;|v_1(\theta)|>1-\varepsilon\}.
\end{align*}
Then $\Theta(u_1,n)\cup\Theta(v_1,n)=[0,2\pi)$.

First consider the case when $\theta\in\Theta(v_1,n)$. 
The curve $\Gamma_n$ consists of the points
$z(\theta)=x(\theta)+iy(\theta)$.
We identify $\mathbb{C}$ with the $\mathbb{R}^2$-space
$\{(x,y)\;|\;x,y\in\mathbb{R}\}$, and identify $z(\theta)$ with $(x(\theta),y(\theta))$.
We study the behavior of the tangent line of the planar curve $\Gamma_n$ at 
$z(\theta)$, when
$\theta$ varies.
By $\Xi(\theta)$ we denote the tangent of the angle of inclination of 
the tangent line at $z(\theta)$.
Then
\begin{align}\label{8-1}
\Xi(\theta)=\frac{y^{\prime}(\theta)}{x^{\prime}(\theta)}
=-\left(\sum_{j=1}^{\infty}jq^j u_j(\theta)\right)\left/
\left(\sum_{j=1}^{\infty}jq^j v_j(\theta)\right)\right..
\end{align}
It is to be noted that the denominator is $qv_1(\theta)+O(q^2)$, so this is non-zero 
for sufficiently small $q$ (that is, sufficiently large $n$), because now we assume
$\theta\in\Theta(v_1,n)$.

We evaluate $\Xi^{\prime}(\theta)$.    First, by differentiation we have
\begin{align}\label{8-2}
\Xi^{\prime}(\theta)=X_1(\theta)+X_2(\theta)+X_3(\theta)+X_4(\theta),
\end{align}
say, where
\begin{align*}
X_1(\theta)&=q v_1(\theta)\left/\left(\sum_{j=1}^{\infty}jq^j v_j(\theta)\right)\right.,\\
X_2(\theta)&=\left(\sum_{j=2}^{\infty} j^2q^j v_j(\theta)\right)\left/
\left(\sum_{j=1}^{\infty}jq^j v_j(\theta)\right)\right.,\\
X_3(\theta)&=(q u_1(\theta))^2\left/\left(\sum_{j=1}^{\infty}jq^j v_j(\theta)\right)^2
\right.,
\end{align*}
and
\begin{align*}
X_4(\theta)=\left(\sum_{\substack{j,k\in\mathbb{N} \\ j+k\geq 3}}jk^2 q^{j+k}u_j(\theta)
u_k(\theta)\right)\left/\left(\sum_{j=1}^{\infty}jq^j v_j(\theta)\right)^2\right..
\end{align*}
We write
\begin{align}\label{8-3}
\sum_{j=1}^{\infty}jq^j v_j(\theta)=qv_1(\theta)(1+Y(\theta)),
\end{align}
where
$$
Y(\theta)=\sum_{j=2}^{\infty}jq^{j-1}\frac{v_j(\theta)}{v_1(\theta)}.
$$
Since $|v_1(\theta)|>1-\varepsilon$, 
using \eqref{estimate-b-c} we have
$$
|Y(\theta)|\leq \frac{4}{1-\varepsilon}\sum_{j=2}^{\infty}q^{j-1}
=\frac{4q}{(1-\varepsilon)(1-q)}=O(q)
$$
(noting $q$ is small).
Therefore
\begin{align}\label{8-4}
&\left(\sum_{j=1}^{\infty}jq^j v_j(\theta)\right)^{-1}
=\frac{1}{q v_1(\theta)}\left(1-\frac{Y(\theta)}{1+Y(\theta)}\right)\\
&\quad=\frac{1}{q v_1(\theta)}+O\left(\frac{1}{q(1-\varepsilon)}\frac{|Y(\theta)|}
{1-|Y(\theta)|}\right)=\frac{1}{q v_1(\theta)}+O(1).\notag
\end{align}
This implies
\begin{align}\label{8-5}
X_1(\theta)=1+O(q).
\end{align}
The numerator of $X_2(\theta)$ can be evaluated, as in \eqref{7-7}, by $O(q^2)$.
Therefore with \eqref{8-4} (whose right-hand side is $O(q^{-1})$) we have
\begin{align}\label{8-6}
X_2(\theta)=O(q^2\cdot q^{-1})=O(q).
\end{align}
As for $X_3(\theta)$, again using $|v_1(\theta)|>1-\varepsilon$ and \eqref{8-4} 
we obtain
\begin{align}\label{8-7}
X_3(\theta)=\frac{u_1(\theta)^2}{v_1(\theta)^2}\left(1-
\frac{Y(\theta)}{1+Y(\theta)}\right)^2
=\frac{u_1(\theta)^2}{v_1(\theta)^2}+O(q).
\end{align}
Lastly, we have
\begin{align}\label{8-8}
X_4(\theta)\ll \sum_{\substack{j,k\in\mathbb{N} \\ j+k\geq 3}}kq^{j+k}\cdot q^{-2}
\ll q,
\end{align}
because
\begin{align*}
\sum_{\substack{j,k\in\mathbb{N} \\ j+k\geq 3}}kq^{j+k}
&=\sum_{j\geq 1}q^j \sum_{k\geq\max\{1,3-j\}}kq^k
=q\sum_{k\geq 2}kq^k+\sum_{j\geq 2}q^j \sum_{k\geq 1}kq^k\\
&=qJ+(q+J)\sum_{j\geq 2}q^j=O(q^3)
\end{align*}
(where $J$ was defined just after \eqref{7-7}).
Collecting \eqref{8-2}, \eqref{8-5}, \eqref{8-6}, \eqref{8-7} and \eqref{8-8}, 
we obtain
\begin{align}\label{8-9}
\Xi^{\prime}(\theta)=1+\frac{u_1(\theta)^2}{v_1(\theta)^2}+O(q).
\end{align}
Note that all the implied constants in the above formulas are absolute.
When $n$ is large, $O(q)$ becomes small, so \eqref{8-9} implies that
$\Xi^{\prime}(\theta)>0$.    That is, if $p_n\in\mathbb{P}_f(\varepsilon)$, $n$ is
sufficiently large, and $\theta\in\Theta(v_1,n)$, then $\Xi(\theta)$ is
monotonically increasing.

In the case when $\theta\in\Theta(u_1,n)$, 
we change the roles of the axes.    That is, now we identify $z(\theta)\in\mathbb{C}$ with
$(-y(\theta),x(\theta))\in\mathbb{R}^2$.
Instead of $\Xi(\theta)$, we consider
$\Xi^*(\theta)=x^{\prime}(\theta)/y^{\prime}(\theta)$.
(The denominator $y^{\prime}(\theta)$ is non-zero for large $n$
because $\theta\in\Theta(u_1,n)$.)
Then $-\Xi^*(\theta)$ is the tangent of the angle of inclination of the tangent line, 
under this new choice of the axes.
We can proceed similarly, and obtain, analogously to \eqref{8-9},
\begin{align}\label{8-10}
(-\Xi^*(\theta))^{\prime}=1+ \frac{v_1(\theta)^2}{u_1(\theta)^2}+O(q),
\end{align}
hence $-\Xi^*(\theta)$ is monotonically increasing when $\theta\in\Theta(u_1,n)$.
Therefore the tangent of the angle of inclination is always increasing, which
implies that the curve $\Gamma_n$ is convex.
\end{proof}

%%%%%%%%%%%%%%%%%%%%%%%%%%%%%%%%%%%%%%%%%%%%%%%%%%%%%%%%%%%%%%%%%%%%%%%%%%%%%%%%%%%%%%%%%%

\end{document}